\documentclass[10pt,reqno]{amsart}
\usepackage{eucal,fullpage,times,amsmath,amsthm,amssymb,mathrsfs,stmaryrd,color,enumerate,accents}
\usepackage[all]{xy}
\usepackage{url}
\usepackage[colorlinks=true,hyperindex, linkcolor=magenta, pagebackref=false, citecolor=cyan]{hyperref}

\newcommand{\cosimp}[3]{\xymatrix@1{#1 \ar@<.4ex>[r] \ar@<-.4ex>[r] & {\ }#2 \ar@<0.8ex>[r] \ar[r] \ar@<-.8ex>[r] & {\ } #3 \ar@<1.2ex>[r] \ar@<.4ex>[r] \ar@<-.4ex>[r] \ar@<-1.2ex>[r] & \cdots }}

\newcommand{\colim}{\mathop{\mathrm{colim}}}

\newcommand{\adjunction}[4]{\xymatrix@1{#1{\ } \ar@<0.3ex>[r]^{ {\scriptstyle #2}} & {\ } #3 \ar@<0.3ex>[l]^{ {\scriptstyle #4}}}}

\renewcommand{\baselinestretch}{1.1}

\begin{document}

\newtheorem{theorem}{Theorem}[section]
\newtheorem*{theorem*}{Theorem}
\newtheorem*{definition*}{Definition}
\newtheorem{proposition}[theorem]{Proposition}
\newtheorem{lemma}[theorem]{Lemma}
\newtheorem{corollary}[theorem]{Corollary}

\theoremstyle{definition}
\newtheorem{definition}[theorem]{Definition}
\newtheorem{question}[theorem]{Question}
\newtheorem{remark}[theorem]{Remark}
\newtheorem{warning}[theorem]{Warning}
\newtheorem{example}[theorem]{Example}
\newtheorem{notation}[theorem]{Notation}
\newtheorem{convention}[theorem]{Convention}
\newtheorem{assumption}[theorem]{Assumption}
\newtheorem{construction}[theorem]{Construction}
\newtheorem{claim}[theorem]{Claim}

\title{On the direct summand conjecture and its derived variant}
\author{Bhargav Bhatt}
\address{Department of Mathematics \\ University of Michigan, Ann Arbor}
\begin{abstract}
Andr\'e recently gave a beautiful proof of Hochster's direct summand conjecture in commutative algebra using perfectoid spaces; his two main results are a generalization of the almost purity theorem (the perfectoid Abhyankar lemma) and a construction of certain faithfully flat extensions of perfectoid algebras where ``discriminants'' acquire all $p$-power roots. 

In this paper, we explain a quicker proof of Hochster's conjecture that circumvents the perfectoid Abhyankar lemma; instead, we prove and use a quantitative form of Scholze's Hebbarkeitssatz (the Riemann extension theorem) for perfectoid spaces. The same idea also leads to a proof of a derived variant of the direct summand conjecture put forth by de Jong. 
\end{abstract}
\maketitle

\section{Introduction}

The first goal of this paper is to give a simpler proof of the following recent result of Andr\'e, settling the direct summand conjecture:

\begin{theorem}[Andr\'e]
\label{thm:DSCIntro}
Let $i:A_0 \hookrightarrow B_0$ be a finite extension of noetherian rings. Assume that $A_0$ is regular. Then the inclusion $i$ is split as an $A_0$-module map.
\end{theorem}

When $A_0$ has characteristic $0$, Theorem~\ref{thm:DSCIntro} is easy to prove using the trace map. When $\dim(A_0) \leq 2$, one can prove Theorem~\ref{thm:DSCIntro} using the Auslander-Buchsbaum formula. Hochster conjectured the general case in 1969, and proved it when $A_0$ has characteristic $p$ in  \cite{HochsterDSC}. The first general result in mixed characteristic was Heitmann's \cite{HeitmannDSC3}, settling the case of dimension $3$. More on the history of this conjecture and its centrality amongst the `homological conjectures' in commutative algebra can be found in \cite{HochsterHomological}. The result above is proven by Andr\'e \cite{AndreDSC} using \cite{AndrePAL}. 

In this paper, we give a proof of Theorem~\ref{thm:DSCIntro} that avoids \cite{AndrePAL} (and is independent of \cite{AndreDSC} in terms of exposition). Our approach also adapts to yield the following derived variant, which was conjectured by Johan de Jong in the course of the author's thesis work \cite{BhattPosCharpdiv, BhattMixedCharpdiv} as a path towards understanding the direct summand conjecture:

\begin{theorem}
\label{thm:DDSCIntro}
Let $A_0$ be a regular noetherian ring. Let $f_0:X_0 \to \mathrm{Spec}(A_0)$ be a proper surjective map. Then the map $A_0 \to R\Gamma(X_0, \mathcal{O}_{X_0})$ splits in the derived category $D(A_0)$.
\end{theorem}

When $A_0$ has characteristic $0$, this result is due to Kov\'acs \cite{KovacsRational}, and is deduced from the fact that $\mathrm{Spec}(A_0)$ has rational singularities; see also \cite[Theorem 2.12]{BhattPosCharpdiv}. The characteristic $p$ case follows from \cite[Theorem 1.4 \& Example 2.3]{BhattPosCharpdiv}. To the best of our knowledge, in mixed characteristic, Theorem~\ref{thm:DDSCIntro} is new even when $\dim(A_0) = 2$.

\begin{remark}
Again, Theorem~\ref{thm:DSCIntro} was proven by Andr\'e. Although it is explained in more detail and with more context in the body of the paper, the main contributions of this paper (as we see it) are:

\begin{enumerate}
\item To clearly explain why Andr\'e's relatively simple flatness lemma from \cite[\S 2.5]{AndreDSC} (reproduced in a slightly cleaned up form in Theorem~\ref{thm:AdjoinRootsDisc}) is the essential new ingredient in the solution of the direct summand conjecture. Indeed, using this lemma, we reprove the conjecture using only the quantitative Hebbarkeitssatz (which is a simple {\em linear} statement about a fixed and explicit system of modules over a perfectoid ring, see Theorem~\ref{thm:QuantRET} and its proof); in contrast, the approach in \cite{AndreDSC} relies on the perfectoid Abhyankar lemma from  \cite{AndrePAL} (which is a deep {\em non-linear} assertion describing an entire class of algebras over a perfectoid ring).

\item To use these techniques to establish derived version of the direct summand conjecture (i.e., Theorem~\ref{thm:DDSCIntro}). Note that Theorem~\ref{thm:DDSCIntro} has a range of geometric implications that are inaccessible from Theorem~\ref{thm:DSCIntro}. For example, it implies that passage to alterations of a regular affine scheme can  often be lossless for coherent cohomological purposes. In fact, even the special case of Theorem~\ref{thm:DDSCIntro} where $f_0$ is birational is a slightly nontrivial assertion about blowups (related to the work in \cite{CRHDIMixedChar}), and completely orthogonal to Theorem~\ref{thm:DSCIntro}.

\end{enumerate}

\end{remark}

\begin{assumption}
In the rest of the introduction, primarily for notational ease, we assume that $A_0 := \widehat{W[x_1,..,x_d]}$ is the $p$-adic completion of a polynomial ring over an unramified dvr $W$ of  mixed characterisitic $(0,p)$; there is a standard reduction of Theorem~\ref{thm:DSCIntro} to mild variants of such an $A_0$, so not much generality is lost. 
\end{assumption}

\subsection{The strategy of Andr\'e's proof}

Andr\'e's proof of Theorem~\ref{thm:DSCIntro} uses perfectoid spaces \cite{ScholzePerfectoidSpaces}. To see why this is natural, we first informally outline the main idea, adapted from \cite{BhattAlmostDSC}, in the special case where $A_0[\frac{1}{p}]  \to B_0[\frac{1}{p}]$ is \'etale. The crucial input is Faltings' almost purity theorem \cite{FaltingsAEE}\footnote{Faltings' theory of almost mathematics, in one incarnation, describes commutative algebra over a ring $V$ equipped with a nonzerodivisor  $f \in V$ that admits arbitrary $p$-power roots. More precisely, one works in the quotient of the abelian category of $V$-modules by the Serre category of all $(f^{\frac{1}{p^\infty}})$-torsion modules. The study of such `almost modules' was inspired by Tate's \cite{TatepDivGrp}, and led Faltings to prove fundamental results in $p$-adic Hodge theory \cite{FaltingspHT,FaltingsCrysCoh,FaltingsVeryRamified,FaltingsAEE}. Following Faltings' work, a systematic investigation of almost mathematics was carried by Gabber and Ramero \cite{GabberRamero,GabberRameroFART}. The relevance of almost mathematics to the direct summand conjecture seems to be first suggested by Roberts \cite{RobertsRootClosure}, \cite[\S 0]{GabberRamero}.}, which asserts: if $A_{\infty,0}$ is the $p$-adic completion of $A_0[x_i^{\frac{1}{p^\infty}}, p^{\frac{1}{p^\infty}}]$, then the $p$-adic completion of the integral closure $B_\infty$ of $B_0$ in $B_0 \otimes_{A_0} A_{\infty,0}[\frac{1}{p}]$ is {\em almost} finite \'etale over $A_{\infty,0}$ with respect to the ideal $(p^{\frac{1}{p^\infty}}) \subset A_{\infty,0}$. Concretely, the algebraic obstructions to $A_{\infty,0} \to B_\infty$ being finite \'etale --- such as the cokernel of the trace map $B_\infty \to A_{\infty,0}$ or the $\mathrm{Ext}^1$-class measuring the failure of $A_{\infty,0} \to B_\infty$ to split --- are killed by $p^{\frac{1}{p^k}}$ for all $k \geq 0$, and are thus quite `small'. The faithful flatness of $A_0 \to A_{\infty,0}$ and the noetherianness of $A_0$ then let one conclude that $A_0 \to B_0$ must actually split. Summarizing, the key ideas are: 
\begin{enumerate}
\item The construction of the faithfully flat extension $A_0 \to A_{\infty,0}$.
\item The almost splitting after base change to $A_{\infty,0}$ coming from the almost purity theorem.
\end{enumerate}
It is now easy to see why perfectoid spaces  provide a natural conceptual home for this proof: the ring $A_{\infty,0}$ in (1) is (integral) perfectoid\footnote{A precise definition is given in \S \ref{notation}. The crucial consequence of a ring $R$ being perfectoid is that the Frobenius $R/p \to R/p$ is surjective, and has a large but controlled kernel. Under suitable completeness and torsionfreeness hypotheses, this leads to the existence of lots of  elements that admit arbitrary $p$-power roots, such as the elements $p$ and $x_i$ in $R = A_{\infty,0}$.}, and the almost purity theorem invoked in (2) is a general fact valid for finite extensions of any perfectoid algebra that are \'etale after inverting $p$ (due to Kedlaya-Liu \cite{KedlayaLiu} and Scholze \cite{ScholzePerfectoidSpaces}). 

Andr\'e's proof of Theorem~\ref{thm:DSCIntro} follows a similar outline to the one sketched above. The first major difference is that $A_{\infty,0}$ is replaced by a larger perfectoid extension $A_\infty$ of $A_{\infty,0}$ coming from the following remarkable construction:

\begin{theorem}[Andr\'e]
\label{thm:AdjoinRootsIntro}
Fix $g \in A_0$. Then there exists a map $A_{\infty,0} \to A_\infty$ of integral perfectoid algebras that is almost faithfully flat modulo $p$ such that the element $g \in A_0$ admits a compatible system of $p$-power roots $g^{\frac{1}{p^k}}$ in $A_\infty$.
\end{theorem} 

Andr\'e's proof of Theorem~\ref{thm:AdjoinRootsIntro} relies crucially on perfectoid geometry, and is explained in \S \ref{sec:AdjoinRootsDisc}. For the application to Theorem~\ref{thm:DSCIntro}, one chooses $g \in A_0$ to be a discriminant, i.e., an element $g$ such that $A_0[\frac{1}{g}] \to B_0[\frac{1}{g}]$ is finite \'etale. The flatness assertions in Theorem~\ref{thm:AdjoinRootsIntro} then reduce us almost splitting the base change $A_\infty \to B_0 \otimes_{A_0} A_\infty$. 

To construct an almost splitting over $A_\infty$, Andr\'e proves a much stronger result, which forms the subject of \cite{AndrePAL}: he generalizes the almost purity theorem to describe extensions of $A_\infty$ that are \'etale after inverting $g$ (almost purity corresponds to $g=p$). The output is {\em roughly} that the integral closure $B_\infty$ of $B_0 \otimes_{A_0} A_\infty$ in $B_0 \otimes_{A_0} A_\infty[\frac{1}{g}]$ is almost finite \'etale over $A_\infty$, where `almost mathematics' is measured with respect to $((pg)^{\frac{1}{p^\infty}}) \subset A_\infty$; the precise statement is more subtle, and we do not formulate it  here as we do not need it.

\subsection{The strategy of our proof}

Our proof uses Theorem~\ref{thm:AdjoinRootsIntro}. Thus, the task is to (almost) split $A_0 \to B_0$ after base change to the ring $A_\infty$ arising from Theorem~\ref{thm:AdjoinRootsIntro}. For this, we again use perfectoid geometry. More precisely, for each $n \geq 1$, the general theory gives us the perfectoid ring $A_\infty \langle \frac{p^n}{g} \rangle$ of bounded functions on the rational subset 
\[U_n := \{x \in X \mid |p^n| \leq |g(x)|\} \]
of the perfectoid space $X$ associated to $A_\infty$. These rings naturally form a projective system as $n$ varies (since $U_n \subset U_{n+1}$), and can be almost described very explicitly: to get $A_\infty \langle \frac{p^n}{g} \rangle$, one formally adjoins $\frac{p^n}{g}$ and its $p$-power roots to $A$.  Their main utility to us\footnote{The idea of using this tower of rings to study $A_\infty$ also comes from \cite{AndrePAL}.} is that $g$ divides $p^n$ in $A_\infty \langle \frac{p^n}{g} \rangle$,  so $A_0 \to B_0$ becomes finite \'etale after base change to $A_\infty \langle \frac{p^n}{g} \rangle[\frac{1}{p}]$ for any $n \geq 0$; the almost purity theorem then kicks in to show that for each $n \geq 0$, the base change of $A_0 \to B_0$ to the perfectoid algebra $A_\infty \langle \frac{p^n}{g} \rangle$ is almost split. To descend the splitting to $A_\infty$, we prove the following quantitative form of Scholze's Riemann extension theorem \cite[Proposition II.3.2]{ScholzeTorsion}. 

\begin{theorem}
\label{thm:RiemannExtensionFinitisticIntro}
Fix an integer $m \geq 0$. The natural map of pro-systems
\begin{equation}
\label{eq:RETIntro}
\{A_\infty/p^m\}_{n \geq 1} \to \{ A_\infty \langle \frac{p^n}{g} \rangle /p^m\}_{n \geq 1} 
\end{equation}
is an almost-pro-isomorphism with respect to $(pg)^{\frac{1}{p^\infty}}$, i.e., for each $k \geq 0$, the pro-system of kernels and cokernels is pro-isomorphic to a pro-system of $(pg)^{\frac{1}{p^k}}$-torsion modules. 
\end{theorem}

\begin{remark}
\label{rmk:ScholzeRH}
On taking limits over $n$ and $m$ in Theorem~\ref{thm:RiemannExtensionFinitisticIntro}, one obtains an almost isomorphism $A_\infty \stackrel{a}{\simeq} \lim A_\infty \langle \frac{p^n}{g} \rangle$, i.e., the following statement from \cite[Proposition II.3.2]{ScholzeTorsion}: any bounded function on the Zariski open set $\{x \in X | \ g(x) \neq 0 \} = \cup_n U_n \subset X$ almost extends to $X$. In other words, this gives a perfectoid analog of the  Riemann extension theorem in complex geometry. A similar result in rigid geometry was proven by Bartenwerfer \cite{Bartenwerfer}. 
\end{remark}

\begin{remark}
Theorem~\ref{thm:RiemannExtensionFinitisticIntro} roughly says that the limiting isomorphism $A_\infty \stackrel{a}{\simeq} \lim A_\infty \langle \frac{p^n}{g} \rangle$ from Remark~\ref{rmk:ScholzeRH} holds true for  `diagrammatic' reasons. Consequently, it remains true after applying $A_\infty$-linear functors, such as $\mathrm{Ext}^i_{A_\infty}(N,-)$ for any $A_\infty$-module $N$, to both sides of \eqref{eq:RETIntro} and then taking limits. The case $i=0$ recovers Scholze's theorem, the case $i=1$ is essential to Theorem~\ref{thm:DSCIntro}, and Theorem~\ref{thm:DDSCIntro} relies on the statement for all $i \geq 0$.
\end{remark}

Using Theorem~\ref{thm:RiemannExtensionFinitisticIntro}, the proof of Theorem~\ref{thm:DSCIntro} proceeds along the lines sketched above, and can thus be summarized as follows: pass from $A_0$ to $A_\infty$ using Theorem~\ref{thm:AdjoinRootsIntro} to ensure this passage is lossless, pass from $A_\infty$ to $A_\infty \langle \frac{p^n}{g} \rangle$ to push all the ramificiation into characteristic $p$, construct an almost splitting over $A_\infty \langle \frac{p^n}{g} \rangle$ using almost purity, and finally take a limit over $n$ to get an almost splitting over $A_\infty$ thanks to Theorem~\ref{thm:RiemannExtensionFinitisticIntro}. In particular, it is exactly the last step (relying on the relatively simple module-theoretic statement in Theorem~\ref{thm:RiemannExtensionFinitisticIntro}) where our approach to Theorem~\ref{thm:DSCIntro} diverges from that of \cite{AndreDSC} (which relies on the  sophisticated perfectoid Abhyankar lemma \cite{AndrePAL}).

To prove Theorem~\ref{thm:DDSCIntro}, we  proceed analogously. First, assume that $f_0$ ramifies only in characteristic $p$, i.e., $f_0[\frac{1}{p}]$ is finite \'etale. Again, it suffices to construct the splitting after going up to a faithfully flat integral perfectoid extension of $A_0$ (such as the ring $A_{\infty,0}$ above). After such a base change, the almost purity theorem and a general vanishing theorem of Scholze settle the question. In general, one first reduces to $f_0$ being generically finite, and finite \'etale after inverting some $g \in A_0$. This case is then deduced from preceding special case using Theorem~\ref{thm:AdjoinRootsIntro} and Theorem~\ref{thm:RiemannExtensionFinitisticIntro}, exactly as was explained above for Theorem~\ref{thm:DSCIntro}.

\subsection{Layout} We begin in \S \ref{sec:AdjoinRootsDisc} by recalling Andr\'e's proof of Theorem~\ref{thm:AdjoinRootsIntro} (in a slightly more general setup). Theorem~\ref{thm:RiemannExtensionFinitisticIntro} is proven in \S \ref{sec:QuantRET}; this depends on the notion of almost mathematics of pro-systems, which is briefly developed in \S \ref{sec:ProAlmostZero}. With these ingredients in place, Theorems~\ref{thm:DSCIntro} and \ref{thm:DDSCIntro} are proven in \S \ref{sec:DSC} and \S \ref{sec:DDSC} respectively.

\subsection{Notation}
\label{notation}
 We freely use the language of perfectoid spaces and almost mathematics. Occasionally, we use almost mathematics with respect to different ideals in the same ring; thus we always specify the relevant ideal, sometimes at the beginning of each section. The letter $K$ denotes a perfectoid field\footnote{At first pass, not much is lost if one simply sets $K := \widehat{\mathbf{Q}_p(p^{\frac{1}{p^\infty}})}$ in characteristic $0$ (with $t=p$), and $K := \widehat{\mathbf{F}_p ((t^{\frac{1}{p^\infty}}))}$ in characteristic $p$.}, and $K^\circ \subset K$ is the ring of integers. We fix an element $t \in K^\circ$ which admits arbitrary $p$-power roots $t^{\frac{1}{p^k}}$, and such that $|t| = |p|$ if $K$ has characteristic $0$. A $K^\circ$-algebra $A$ is called {\em integral perfectoid} if it is flat, $t$-adically complete, satisfies\footnote{Concretely, the assumption $A = A_*$ means that if $f \in A[\frac{1}{t}]$ is such that $t^{\frac{1}{p^k}} \cdot f \in A$ for all $k \geq 0$, then $f \in A$. This condition is not really serious and can often be ignored: if $A$ is a $t$-adically complete and flat $K^\circ$-algebra with Frobenius inducing an almost isomorphism $A/t^{\frac{1}{p}} \simeq A/t$, then $A' := A_* :=  \mathrm{Hom}_{K^\circ}(t^{\frac{1}{p^\infty}}, A)$ is an integral perfectoid $K^\circ$-algebra in the sense introduced above by \cite[Lemma 5.6]{ScholzePerfectoidSpaces}, and $A \to A_*$ is an almost isomorphism.} $A = A_*$, and satisfies the following: Frobenius induces an isomorphism $A/t^{\frac{1}{p}} \simeq A/t$. The category of such algebras is equivalent to usual category of perfectoid $K$-algebras by \cite[Theorem 5.2]{ScholzePerfectoidSpaces}; the functors are $A \mapsto A[\frac{1}{t}]$ and $R \mapsto R^\circ$ respectively. 

\subsection*{Acknowledgements} This paper is obviously inspired by Andr\'e's preprints \cite{AndrePAL,AndreDSC}; I thank heartily him for sharing them, and for his kind words about this paper. I am also very grateful to Peter Scholze for patiently explaining basic facts about perfectoid spaces, for discussions surrounding Theorem~\ref{thm:AdjoinRootsIntro}, and for convincing me that `almost-pro-isomorphism' is a better name than `pro-almost-isomorphism' in \S \ref{sec:ProAlmostZero}. I am equally indebted to my former PhD advisor Johan de Jong for talking through the contents of this manuscript, and his very prescient suggestion (slightly over seven years ago!) that I pursue the mathematics surrounding Theorem~\ref{thm:DDSCIntro}.  Finally, I thank Brian Conrad, Ray Heitmann, Mel Hochster, Kiran Kedlaya, Linquan Ma, Peter Scholze, Karl Schwede and an anonymous referee for many useful comments on the first version of this paper. I was partially supported by NSF Grant DMS \#1501461 and a Packard fellowship during the preparation of this work.

\section{Adjoining roots of the discriminant}
\label{sec:AdjoinRootsDisc}

\begin{notation}
Let $A$ be an integral perfectoid $K^\circ$ algebra. Fix $g \in A$. Set $X := \mathrm{Spa}(A[\frac{1}{t}],A)$ and $Y := \mathrm{Spa}(A \langle T^{\frac{1}{p^\infty}} \rangle[\frac{1}{t}], A \langle T^{\frac{1}{p^\infty}} \rangle)$; these are perfectoid spaces. All occurrences of almost mathematics in this section are with respect to $t^{\frac{1}{p^\infty}}$.
\end{notation}

The main goal of this section is to construct an almost faithfully flat extension $A \to A_\infty$ of perfectoid algebras such that $g$ acquires arbitrary $p$-power roots in $A_\infty$. For this, we essentially set $T=g$ in $A\langle T^{\frac{1}{p^\infty}} \rangle$. More precisely,  to get a perfectoid algebra, we approximate bounded functions on the Zariski closed space 
\[ Z := V(T-g) := \{y \in Y | T(y) = g(y)\} \subset Y\]
using bounded functions on rational open neighbourhoods 
\[ Y\langle \frac{T-g}{t^\ell} \rangle := \{y \in Y \mid |T(y) - g(y)| \leq |t^\ell|\} \subset Y\] 
of $Z$ for varying integers $\ell$. 

\begin{definition}
Set $A_\infty$ to be the integral perfectoid\footnote{The ring $A_\infty$ defined here might not be integral perfectoid, but is almost isomorphic to one (by passing to $(A_\infty)_{\ast}$), so we ignore the distinction.} ring of functions on the Zariski closed subset of $Y$ defined by the ideal $(T - g)$, in the sense of \cite[\S II.2]{ScholzeTorsion}. Explicitly, we have
\[ A_\infty = \widehat{\colim_{\ell \in \mathbf{N}}} \ B_\ell \quad \text{where} \quad B_\ell := \mathcal{O}_Y^+(Y\langle \frac{T-g}{t^\ell} \rangle), \]  
and the completion appearing on the left is $t$-adic. 
 \end{definition}

 Note that $T = g$ in $A_\infty$ as $(T-g)$ is divisible by $t^\ell$ in $B_\ell$, and thus in $A_\infty$, for all $\ell$. Thus, $g$ has a distinguished system of $p$-power roots $g^{\frac{1}{p^k}} := T^{\frac{1}{p^k}}$ in $A_\infty$. The main theorem is (see \cite[\S 2.5]{AndreDSC}):

\begin{theorem}[Andr\'e]
\label{thm:AdjoinRootsDisc}
For each $\ell > 0$, the map $A \to B_\ell$ is almost faithfully flat modulo $t$. Consequently, the map $A \to A_\infty$ is almost faithfully flat modulo $t$. 
\end{theorem}

\begin{proof}
It is enough to show the first statement modulo $t^\epsilon$ for some $\epsilon > 0$. The approximation lemma for perfectoid spaces (as in \cite[Corollary 3.6.7]{KedlayaLiu} or \cite[Corollary 6.7]{ScholzePerfectoidSpaces}) gives an $f \in \big(A\langle T^{\frac{1}{p^\infty}} \rangle\big)^\flat$ such that 
\begin{enumerate}
\item $f^\sharp \equiv T - g \mod t^{\frac{1}{p}}$.
\item We have an equality $Y \langle \frac{T-g}{t^\ell} \rangle = Y \langle \frac{f^\sharp}{t^\ell} \rangle$ of subsets of $Y$.
\end{enumerate}
The explicit description of $\mathcal{O}_Y^+(Y \langle \frac{f^\sharp}{t^\ell} \rangle)$ from \cite[Lemma 6.4]{ScholzePerfectoidSpaces} identifies $B_\ell$ (almost)  with the $t$-adic completion of
\begin{equation}
\label{eq:PresentationBoundedFunctions}
\colim_k \Big(A\langle T^{\frac{1}{p^\infty}}\rangle[u^{\frac{1}{p^k}}] / \big((u \cdot t^\ell)^{\frac{1}{p^k}} - (f^\sharp)^{\frac{1}{p^k}}\big)\Big)    = A\langle T^{\frac{1}{p^\infty}}\rangle[u^{\frac{1}{p^\infty}}] / \big(\forall k:  (u \cdot t^\ell)^{\frac{1}{p^k}} - (f^\sharp)^{\frac{1}{p^k}}\big). 
\end{equation}
Thus, it is enough to show that the $A$-algebra
\[ C_{\ell,k} := A\langle T^{\frac{1}{p^\infty}}\rangle[u^{\frac{1}{p^\infty}}] / \big((u \cdot t^\ell)^{\frac{1}{p^k}} - (f^\sharp)^{\frac{1}{p^k}}\big)\]
is faithfully flat over $A$ after reduction modulo $t^\epsilon$ for some $\epsilon$. We take $\epsilon = \frac{1}{p^{k+1}}$. For this choice, we have $t^{\frac{\ell}{p^k}} \equiv 0 \mod t^\epsilon$, so the relation above simplifies to $(f^\sharp)^{\frac{1}{p^k}} = 0$ modulo $t^\epsilon$. As $f^\sharp \equiv T-g \mod t^{\frac{1}{p}}$, the $k$-fold Frobenius identifies the $A/t^\epsilon$-algebra $C_{\ell,k}/t^\epsilon$ with the $A/t^{\frac{1}{p}}$-algebra $A[T^{\frac{1}{p^\infty}},u^{\frac{1}{p^\infty}}]/(t^{\frac{1}{p}}, T-g)$. The latter is faithfully flat (even free) over $A/t^{\frac{1}{p}}$, so the claim follows.
%
\end{proof}

\begin{remark}
Theorem~\ref{thm:AdjoinRootsDisc} is proven in \cite{AndreDSC} under a more restrictive setup (but with a stronger conclusion). I am grateful to Scholze for pointing out that the same proof goes through in the above generality.
\end{remark}

\begin{remark}
One might worry that the presentations from \cite[Lemma 6.4]{ScholzePerfectoidSpaces} used above are only valid in the non-derived sense, and thus do not play well with reduction modulo $t$ or $t$-adic completion. More precisely, one may ask if \eqref{eq:PresentationBoundedFunctions} is also true if one imposes the corresponding relations in the derived sense (i.e., one works with the corresponding Koszul complexes). While answering this question is not necessary for our purposes, the answer is indeed `yes', and we record it here for psychological comfort, especially since such presentations are also important later. 

\begin{lemma}
Let $A$ be an integral perfectoid $K^\circ$-algebra. Choose $f_1,...,f_n,g \in A^\flat$, and set $B$ to be the direct limit of the Koszul complexes $\mathrm{Kos}(A[T_i^{\frac{1}{p^\infty}}] ; (g^\sharp \cdot T_i)^{\frac{1}{p^m}} - (f_i^\sharp)^{\frac{1}{p^m}})$. Then the Koszul complex $\mathrm{Kos}(B; t)$ is almost discrete. Thus, the derived $t$-adic completion of $B$ is almost isomorphic to the perfectoid algebra $A \langle \frac{f_1}{g},...,\frac{f_n}{g} \rangle$.
\end{lemma}
\begin{proof}
Note that $A[T_i^{\frac{1}{p^\infty}}]$ has no $t$-torsion. Thus, the complex $\mathrm{Kos}(B; t)$ is identified with 
\[ M := \colim_m \Big(\mathrm{Kos}(A/t[T_i^{\frac{1}{p^\infty}}]; (g^\sharp \cdot T_i)^{\frac{1}{p^m}} - (f_i^\sharp)^{\frac{1}{p^m}})\Big)\]
since, at level $m$, freely imposing the relations $t=0$ and $(g^\sharp \cdot T_i)^{\frac{1}{p^m}} - (f_i^\sharp)^{\frac{1}{p^m}} = 0$ in the derived sense on the ring $A[T_i^{\frac{1}{p^\infty}}]$ can be done in any order. But now $M$ looks the same for both $A$ and $A^\flat$, so we may assume that $A$ has characteristic $p$ (and so $f_i = f_i^\sharp$, $g = g^\sharp$). In this case, $M$ identifies with $\mathrm{Kos}(R; t)$, where 
 \[ R :=  \colim_m \Big(\mathrm{Kos}(A[T_i^{\frac{1}{p^\infty}}]; (g \cdot T_i)^{\frac{1}{p^m}} - (f_i)^{\frac{1}{p^m}})\Big).\]
But $R$ is discrete: it is the perfection of the derived ring $\mathrm{Kos}(A[T_i^{\frac{1}{p^\infty}}]; g \cdot T_i - f)$, which is always discrete by \cite[Lemma 3.16 or Proposition 5.6]{BhattScholzeWitt}. As $M \simeq \mathrm{Kos}(R; t)$, we are reduced to showing that the $t$-torsion of $R$ is almost zero.  But this follows from perfectness: if $\alpha \in R$ and $t \cdot \alpha = 0$, then $t \cdot \alpha^{p^n} = 0$ for all $n \geq 0$, which, by perfectness, gives $t^{\frac{1}{p^n}} \cdot \alpha = 0$ for all $n \geq 0$, so $\alpha$ is almost zero.
\end{proof}
\noindent In particular, all operations in the proof of Theorem~\ref{thm:AdjoinRootsDisc}  can be interpreted in the derived sense. 
\end{remark}

\begin{remark}
One may upgrade the above techniques to show the following (see \cite[Corollary 9.4.7]{BhattPerfSpaceNotes}): for any integral perfectoid $K^\circ$-algebra $A$, there exists a {\em functorial} map $A \to B(A)$ of integral perfectoid $K^\circ$-algebras that is almost faithfully flat modulo $t$ such that $B(A)$ is absolutely integrally closed, i.e., each monic polynomial has a solution. In particular, any $b \in B(A)$ admits a compatible system $\{b^{\frac{1}{p^n}}\}_{n \geq 1}$ of $p$-power roots.
\end{remark}

\section{Almost-pro-zero modules}
\label{sec:ProAlmostZero}

We introduce the relevant notion of almost mathematics in the pro-category necessary for Theorem~\ref{thm:RiemannExtensionFinitisticIntro}.

\begin{notation}
Let $A$ be a ring equipped with a nonzerodivisor $t$ together with a specified collection $\{t^{\frac{1}{p^k}}\}$ of compatible $p$-power roots. All occurrences of almost mathematics in this section are with respect to $t^{\frac{1}{p^\infty}}$.
\end{notation}

There is an intrinsic notion of almost mathematics of pro-$A$-modules: one might simply work with pro-objects in the almost category. For example, a projective system $\{M_n\}_{n \geq 1}$ of $A$-modules is `almost-zero' as a pro-object if for any $n \geq 1$, there exists some $m = m(n) \geq n$ such that the map $M_m \to M_n$ has image annihilated by $t^{\frac{1}{p^k}}$ for all $k$. This intrinsic notion is too strong for our purposes, and we use the following weakening, where $m$ depends on $k$:

\begin{definition}
A pro-$A$-module $\{M_n\}_{n \geq 1}$ is said to be {\em almost-pro-zero} if for any $k \geq 0$ and any $n \geq 1$, there exists some $m = m(n,k) \geq n$ such that $\mathrm{im}(M_m \to M_n)$ is killed by $t^{\frac{1}{p^k}}$; equivalently, for each $k \geq 0$, the map $\{M_n[t^{\frac{1}{p^k}}]\}_{n \geq 1} \to \{M_n\}_{n \geq 1}$ is a pro-isomorphism in the usual sense. A map of pro-objects in $D^b(A)$ is said to be an {\em almost-pro-isomorphism} if the cohomology groups of cones form an almost-pro-zero system.
\end{definition}

We begin with an example illustrating the novel features of this notion:

\begin{example}
Consider the system $\{M_n\}$ where $M_n = A/(t^{\frac{1}{p^n}})$, and $M_{n+1} \to M_n$ is the injective map defined by $1 \mapsto t^{\frac{1}{p^n} - \frac{1}{p^{n+1}}}$. Then $\{M_n\}$ is almost-pro-zero (in fact, each $M_m$ is killed by $t^{\frac{1}{p^k}}$ for $m \geq k$), even though the corresponding pro-object of the almost category is not zero.
\end{example}

The next few lemmas record the stability properties of this notion:

\begin{lemma}
\label{lem:PAIAlmostZero}
If $\{M_n\}_{n \geq 1}$ is an almost-pro-zero pro-$A$-module, then the complex $R\lim(\{M_n\}_{n \geq 1})$ is almost zero, i.e., it has almost zero cohomology groups. 
\end{lemma}
\begin{proof}
Fix  $k \geq 0$. Then the inclusion $\{M_n[t^{\frac{1}{p^k}}]\}_{n \geq 1} \to \{M_n\}_{n \geq 1}$ is a pro-isomorphism, so both sides have the same $R\lim$. In particular, the cohomology groups of $R\lim(\{M_n\}_{n \geq 1})$ are killed by $t^{\frac{1}{p^k}}$.
\end{proof}

\begin{lemma}
\label{lem:PAIAlmostIso}
If $\{N_n\}_{n \geq 1} \to \{M_n\}_{n \geq 1}$ is an almost-pro-isomorphism in $D^b(A)$, then $R\lim(\{N_n\}_{n \geq 1}) \to R\lim(\{M_n\}_{n \geq 1})$ is an almost isomorphism.
\end{lemma}
\begin{proof}
This follows by applying Lemma~\ref{lem:PAIAlmostZero} to the cone.
\end{proof}

\begin{lemma}
\label{lem:PAIAlinear}
If $\{M_n\}_{n \geq 1}$ is an almost-pro-zero pro-$A$-module, and $F:\mathrm{Mod}_A \to \mathrm{Mod}_A$ is an $A$-linear functor, then $\{F(M_n)\}_{n \geq 1}$ is also almost-pro-zero.
\end{lemma}
\begin{proof}
Fix  $k \geq 0$, $n \geq 1$. Then $M_m \to M_n$ factors over $M_n[t^{\frac{1}{p^k}}] \subset M_n$ for some $m \geq n$. But then $F(M_m) \to F(M_n)$ factors over $F(M_n[t^{\frac{1}{p^k}}]) \to F(M_n)$, and hence over $F(M_n)[t^{\frac{1}{p^k}}] \hookrightarrow F(M_n)$, by the $A$-linearity of $F$. 
\end{proof}

\section{A quantitative form of the Riemann extension theorem}
\label{sec:QuantRET}

\begin{notation}
\label{not:QuantRET}
Let $A$ be an integral perfectoid $K^\circ$-algebra with associated perfectoid space $X := \mathrm{Spa}(A[\frac{1}{t}],A)$. Fix an element $g \in A$ that admits a compatible system of $p$-power roots $g^{\frac{1}{p^k}}$. Assume\footnote{This assumption is not actually necessary, and can be dropped {\em a posteriori}; see Remark~\ref{rmk:RETnzd}. It is also harmless in applications.} that $g$ is a nonzerodivisor modulo $t^m$ in the almost sense (with respect to $t^{\frac{1}{p^\infty}}$). 
\end{notation}

In this section, we prove Theorem~\ref{thm:RiemannExtensionFinitisticIntro}. Thus, we study the rings $A \langle \frac{t^n}{g} \rangle := \mathcal{O}_X^+(X\langle \frac{t^n}{g} \rangle)$ and their variation with $n$. More precisely, we show the following quantitative form of Scholze's Hebbarkeitssatz \cite[Proposition II.3.2]{ScholzeTorsion}:

\begin{theorem}
\label{thm:QuantRET}
For each $m \geq 0$, consider the natural projective system of maps
\[ \{ f_n: A/t^m  \to A \langle \frac{t^n}{g} \rangle/t^m\}_{n \geq 1}\]
in almost mathematics with respect to $t^{\frac{1}{p^\infty}}$. Then we have:
\begin{enumerate}
\item Each $\mathrm{ker}(f_n)$ is almost zero.
\item The pro-system $\{\mathrm{coker}(f_n)\}_{n \geq 1}$ is uniformly almost-pro-zero with respect to $g^{\frac{1}{p^\infty}}$, i.e., for any $k \geq 0$, there exists some $c \geq 0$ such that for any $n \geq 0$, the image of the $c$-fold transition map $\mathrm{coker}(f_{n+c}) \to \mathrm{coker}(f_n)$ is killed by $g^{\frac{1}{p^k}}$. (In fact, $c = p^k m$ works.)
\end{enumerate}
In particular, the projective system $\{f_n\}_{n \geq 1}$ is an almost-pro-isomorphism with respect to $(tg)^{\frac{1}{p^\infty}}$.
\end{theorem}
\begin{proof}
Fix some integer $m \geq 0$. We use the explicit presentations for $A\langle \frac{t^n}{g} \rangle/t^m$ coming from the perfectoid theory. By \cite[Lemma 6.4]{ScholzePerfectoidSpaces}, there is almost isomorphism (with respect to $t^{\frac{1}{p^\infty}}$)
\[ M_n :=  A[u_n^{\frac{1}{p^\infty}}]/\big(t^m, \forall k: (u_n \cdot g)^{\frac{1}{p^k}} - t^{\frac{n}{p^k}}\big)  \stackrel{a}{\simeq} A\langle \frac{t^n}{g} \rangle/t^m.\]
defined by  viewing $u_n^{\frac{1}{p^k}}$ as the function $(\frac{t^n}{g})^{\frac{1}{p^k}}$. It is thus enough show the assertions in the theorem for the pro-system
\[ \{f_n: A/t^m \to M_n\}_{n \geq 1}\]
of obvious maps. As $g$ is a nonzerodivisor modulo $t^m$, the same holds true for $g^{\frac{1}{p^k}}$. It is then easy see that each $f_n$ is injective, so the kernels are $0$ on the nose. For the cokernels, fix some $k \geq 0$. We shall show that any $c \geq p^k \cdot m$ works, i.e., for such $c$, the element $g^{\frac{1}{p^k}} \cdot u_{n+c}^e \in M_{n+c}$ maps into $A/t^m \subset M_n$ under the $c$-fold transition map $M_{n+c} \to M_n$ for all exponents $e \in \mathbf{N}[\frac{1}{p}]$. By construction, the transition map carries $g^{\frac{1}{p^k}} \cdot u_{n+c}^e$ to
\[ g^{\frac{1}{p^k}} \cdot t^{ce} \cdot u_n^e \in M_n,\]
so we must show this last expression lies in $A/t^m \subset M_n$ for $c \geq p^k m$ and all $e \in \mathbf{N}[\frac{1}{p}]$. There are two cases:
\begin{itemize}
\item If $e \geq \frac{1}{p^k}$, then $t^m \mid t^{ce}$ as $c \geq p^k m$, so the above expression is zero as we work modulo $t^m$.
\item If $e < \frac{1}{p^k}$, then the above expression can be written as
\[ g^{\frac{1}{p^k}} \cdot t^{ce} \cdot u_n^e  = g^{\frac{1}{p^k} - e} \cdot g^{e} \cdot t^{ce} \cdot u_n^e = g^{\frac{1}{p^k} - e}  \cdot t^{ce} \cdot (g \cdot u_n)^e = g^{\frac{1}{p^k} - e}  \cdot t^{ce} \cdot t^{ne} = g^{\frac{1}{p^k} -e} \cdot t^{(n+c)e} \in A/t^m \subset M_n,  \] 
as wanted. \qedhere
\end{itemize}
\end{proof}

\begin{remark}
Theorem~\ref{thm:QuantRET} shows that the map $\{ A/t^m \}_{n \geq 1} \to \{A \langle \frac{t^n}{g} \rangle/t^m\}_{n \geq 1}$ is a {\em uniform} almost-pro-isomorphism with respect to $(tg)^{\frac{1}{p^\infty}}$, i.e., the constant $c$ appearing in the theorem is independent of $n$. It formally follows that for any $A/t^m$-complex $K$, the kernel and cokernel pro-systems of the induced 
\[ \{ H^i(K) \}_{n \geq 1} \to \{H^i(K \otimes^L_{A/t^m} A \langle \frac{t^n}{g} \rangle/t^m)\}_{n \geq 1}\]
are both uniformly almost-pro-zero in the preceding sense and with the same implicit constants. In particular, when applied to Koszul complexes arising from regular sequences of elements in $A/t^m$, we learn that the homology of the corresponding pro-system of Koszul complexes on $\{A \langle \frac{t^n}{g} \rangle/t^m\}$ is uniformly almost-pro-zero in nonzero degrees.
\end{remark}

\begin{remark}
\label{rmk:RETnzd}
The assumption that $g$ is a nonzerodivisor modulo $t^m$ in Notation~\ref{not:QuantRET} can be dropped without affecting the conclusion of the final statement of Theorem~\ref{thm:QuantRET}. Indeed, consider first the universal case $R := K^\circ \langle T^{\frac{1}{p^\infty}} \rangle$ with $g=T$. This falls under the case that is already treated, so we have an almost-pro-isomorphism
\[ \{R/t^m\}_{n \geq 1} \to \{R \langle \frac{t^n}{T} \rangle/t^m\}_{n \geq 1}\]
with respect to $(tT)^{\frac{1}{p^\infty}}$. For general $A$ and $g$, there is a unique map $R \to A$ carrying $T^{\frac{1}{p^k}}$ to $g^{\frac{1}{p^k}}$ for all $k$. By base change, we have an almost-pro-isomorphism
\[ \{A/t^m\}_{n \geq 1} \to \{R \langle \frac{t^n}{T} \rangle \otimes^L_R A/t^m\}_{n \geq 1}\]
with respect to $(tg)^{\frac{1}{p^\infty}}$. The explicit description of \cite[Lemma 6.4]{ScholzePerfectoidSpaces} shows that
\[R \langle \frac{t^n}{T} \rangle \otimes_R A/t^m \stackrel{a}{\simeq} A\langle \frac{t^n}{g} \rangle/t^m.\]
In particular, applying $H^0$ to the almost-pro-isomorphism above gives the desired statement.
\end{remark}

\section{The direct summand conjecture}
\label{sec:DSC}

In this section, we prove Theorem~\ref{thm:DSCIntro}. We begin by collecting some preliminaries that shall be useful in the proof. The following proposition is borrowed from \cite[Lemma 4.30 and Remark 4.31]{BhattMorrowScholze}, and is presumably well-known:

\begin{proposition}
\label{prop:FaithfulFlatnessNoetherian}
Let $A \to B$ be a map of commutative rings with $A$ is noetherian. Assume that there exists some  $\pi \in A$ such that both $A$ and $B$ are $\pi$-torsionfree and $\pi$-adically complete, and $A/\pi \to B/\pi$ is (faithfully) flat. Then $A \to B$ is (faithfully) flat.
\end{proposition}
\begin{proof}
For flatness: we must check that $M \otimes_A^L B$ lies in $D^{\geq 0}$ for any finitely generated $A$-module $M$. As $A$ is noetherian, we can choose a resolution $P^\bullet \to M$ with each $P^i$ being finite free. The complex $M \otimes_A^L B$ is then computed by $P^\bullet \otimes_A B$. As $B$ is $\pi$-adically complete, we have $P^\bullet \otimes_A B \simeq \lim_n P^\bullet \otimes_A B/\pi^n$ at the level of complexes. The transition maps in the system on the right are termwise surjective, so we can write this more intrinsically as
\[ M \otimes_A^L B \simeq R\lim_n (M \otimes_A^L B/\pi^n) \simeq R\lim_n ((M \otimes_A^L A/\pi^n) \otimes_{A/\pi^n}^L B/\pi^n).\]
As $M$ is finitely generated, the pro-$A$-complex $\{M \otimes_A^L A/\pi^n\}$ is pro-isomorphic to $\{M/\pi^n\}$: the obstruction is the pro-system $\{M[\pi^n]\}$, which is pro-zero as the $\pi^\infty$-torsion of $M$ is bounded by finite generation. Thus, we obtain
\[ M \otimes_A^L B \simeq R\lim_n (M/\pi^n \otimes_{A/\pi^n}^L B/\pi^n).\]
 As $A/\pi \to B/\pi$ is flat, the same holds true for $A/\pi^n \to B/\pi^n$ as $\pi$ is a nonzerodivisor on both $A$ and $B$. In particular, the terms showing up inside the limit lie in $D^{\geq 0}$, so the same holds true for the limit, as wanted.
 
 For faithful flatness, we must check that $\mathrm{Spec}(B) \to \mathrm{Spec}(A)$ is surjective if $A \to B$ is flat and $A/\pi \to B/\pi$ is faithfully flat. As the image is stable under generalizations by flatness, it suffices to check that all closed points lie in the image; equivalently, we must show that $A/\mathfrak{m} \otimes_A B \neq 0$ for any maximal ideal $\mathfrak{m}$ in $A$. But $\pi \in \mathfrak{m}$ as $A$ is $\pi$-adically complete, so $A/\mathfrak{m} \otimes_A B \simeq A/\mathfrak{m} \otimes_{A/\pi} B/\pi$, which is nonzero by faithful flatness of $A/\pi \to B/\pi$.
\end{proof}

Next, we explain why regular local rings admit faithfully flat covers by perfectoids. 

\begin{proposition}
\label{prop:RegularPerfectoidFaithfullyFlat}
Let $A_0$ be a $p$-torsionfree noetherian regular local ring whose residue characteristic is $p$. Then there exists a map $A_0 \to A$ such that
\begin{enumerate}
\item The ring $A$ admits the structure of an integral perfectoid $K$-algebra for $K = \widehat{\mathbf{Q}_p(p^{\frac{1}{p^\infty}})}$.
\item The map $A_0 \to A$ is ``almost faithfully flat" in the following sense: for any $A_0$-module $M$, we have
\begin{enumerate}
\item $\mathrm{Tor}_i^{A_0}(M,A)$ is almost zero for $i > 0$.
\item If $M \otimes_{A_0} A$ is almost zero, then $M = 0$.
\end{enumerate}
\end{enumerate}
\end{proposition}

In the unramified case, we can also arrange for $A_0 \to A$ to be faithfully flat. The proof below shows that it is possible to achieve the same in general provided we make either one of the following modifications: (a) relax (1) above to only requiring either that $A$ is an integral perfectoid ring in a generalized sense (i.e., one that does not necessarily contain a perfectoid field, as elaborated in the proof below), or (b) only require $A$ to be a $p$-adically complete and $p$-torsionfree $K^\circ$-algebra that is almost isomorphic to an integral perfectoid $K^\circ$-algebra. Related constructions occur in \cite[Proposition 4.9]{ShimomotoAlmostPurity} or \cite[Example 3.4.6 (3)]{AndrePAL}.

\begin{proof}
We are free to replace $A_0$ by noetherian regular local rings that are faithfully flat over it. Thus, we may assume that $A_0$ is complete for the topology defined by powers of the maximal ideal, and has an algebraically closed residue field $k$. Let $W = W(k)$ be the Witt vectors of $k$, and write $\mathfrak{m} \subset A_0$ for the maximal ideal. Write $d = \dim(A_0)$.

Assume $p \notin \mathfrak{m}^2$ (which is the so-called {\em unramified} case). Then $p$ is part of a basis of $\mathfrak{m}/\mathfrak{m}^2$, and thus $A_0$ is isomorphic to $W \llbracket x_2,...,x_d \rrbracket$. In this case, we may simply take $A$ to be the $p$-adic completion of $A_0[p^{\frac{1}{p^\infty}},x_i^{\frac{1}{p^\infty}}]$. In this case, the map $A_0 \to A$ is faithfully flat by Proposition~\ref{prop:FaithfulFlatnessNoetherian} (and thus also almost faithfully flat by the argument given at the end of this proof for the ramified case). Note that this case suffices Theorem~\ref{thm:DSCIntro} by \cite[Theorem 6.1]{HochsterCanonical}.

Assume $p \in \mathfrak{m}^2$ (which is the so-called {\em ramified} case). By choosing $d$ generators for $\mathfrak{m}$, we obtain a surjection $\psi:P_0 := W \llbracket x_1,...,x_d \rrbracket \to A_0$. Using the regularity of $A_0$ and the assumption $p \in \mathfrak{m}^2$, it is easy to see that $\ker(\psi)$ is generated by an element of the form $p - f$ where $f = f(x_i) \in (p,x_1,...,x_d)^2$ is a power series. Moreover, as $A_0$ is $p$-torsionfree and $p$-adically complete, we may also conclude that $p \nmid f$ and $f$ has no constant term. Now write $P_m = P_0[x_i^{\frac{1}{p^m}}]$, and consider the ring $A'$ obtained as the $p$-adic completion of $(\colim_m P_m) \otimes_{P_0} A_0 \simeq \colim_m P_m/(p-f)$. As $P_0 \to P_m$ is faithfully flat, it is easy to see that $A_0 \to A'$ is also faithfully flat. Moreover, the element $g = \sigma^{-1}(f)(x_i^{\frac{1}{p}}) \in P_1$ satisfies $g^p = f + ph$ for some $h \in P_1$; here $\sigma$ is the (unique) lift of the Frobenius automorphism of $k$ to $W$, and $\sigma^{-1}(f)$ is the power series obtained by applying $\sigma^{-1}$ to the coefficients of $f$. As $f$ and $g$ have no constant terms, nor does $h$. In $A'$, this gives $g^p = p + ph = p(1+h) = pu$ for some unit $u \in A'$. In particular, the ring $A'$ equipped with the $p$-adic topology is {\em integral perfectoid in a generalized sense}, i.e., the topological ring $A'[\frac{1}{p}]$ (topologized by making $p^n A'$ a neighbourhood basis of $0$) is a perfectoid Tate ring in the sense of \cite[Definition 6.1.1]{ScholzeWeinsteinBerkeley}, and $A'$ is a ring of integral elements in $A'[\frac{1}{p}]$. The proof of \cite[see Lemma 6.2.2]{ScholzeWeinsteinBerkeley} gives an element $\pi \in A'$ admitting a compatible system of $p$-power roots such that $\pi^p = p v$ for some unit $v \in A'$. Henceforth, almost mathematics over $A'$ is measured with respect to $(\pi^{\frac{1}{p^\infty}})$; this also coincides with the ideal$\sqrt{(p)}$, and is thus independent of the choice of $\pi$ or its roots.

Now the theory of perfectoid spaces extends to the generalized setting, see \cite[\S 6]{ScholzeWeinsteinBerkeley}, \cite[\S 16]{GabberRameroFART}, and \cite[\S 3]{KedlayaLiu}. Consider the ind-(finite \'etale) extension $A'[\frac{1}{p}] \to A'[\frac{1}{p}, v^{\frac{1}{p^\infty}}]$ obtained by formally extracting $p$-power roots of $v$ from $A'[\frac{1}{p}]$. By the almost purity theorem  \cite[Theorem 7.9 (iii)]{ScholzePerfectoidSpaces}, the $\pi$-adic completion $A$ of the integral closure of $A'$ in this extension of $A'[\frac{1}{p}]$ is the $\pi$-adic completion of an ind-(almost finite \'etale) extension of $A'$. In particular, $A$ is an integral perfectoid ring in the generalized sense, there is no $\pi$-torsion in $A$, and the map $A' \to A$ is almost faithfully flat modulo $\pi$. By construction, the element $p = \pi^p v^{-1} \in A'$ admits a compatible system of $p$-power roots, so $A'$ can also be viewed as an integral perfectoid algebra over $K^\circ$ for $K = \widehat{\mathbf{Q}_p(p^{\frac{1}{p^\infty}})}$ in the sense used elsewhere in this article (at least after application of $(-)_*$, which is harmless for for our purposes). Note that ideals $(p^{\frac{1}{p^\infty}})$ and $(\pi^{\frac{1}{p^\infty}})$ in $A$ are identical (and both coincide with $\sqrt{(p)}$), so there is a natural notion of almost mathematics over $A$.

It remains to check that the composite map $A_0 \to A$ satisfies (2). This map factorizes as
\[ A_0 \xrightarrow{a} A' \xrightarrow{b} A_{!!} \xrightarrow{c} A\]
where $(-)_{!!}$ is defined as in \cite[Definition 2.2.23]{GabberRamero}. By construction, the map $a$ is faithfully flat and the map $c$ is an injective almost isomorphism. In particular, $A_{!!}$ is $p$-adically complete and $p$-torsionfree as $A$ is so. As the formation of $(-)_{!!}$ commutes with reduction modulo $p^m$ (see \cite[Remark 2.2.28 (ii)]{GabberRamero}), it follows from the almost faithful flatness modulo $p^m$ of $A' \to A$ and \cite[Remark 3.1.3 (ii)]{GabberRamero} that $b$ is faithfully flat modulo $p^m$ for any $m \geq 0$. The composite $b \circ a$ is then faithfully flat by Proposition~\ref{prop:FaithfulFlatnessNoetherian}. As $c$ is an almost isomorphism, this verifies (a) in (2).

For (b) in (2), say $M$ is an $A_0$-module with $M \otimes_{A_0} A$ almost zero. As $c$ is an almost isomorphism, this is equivalent to asking $M \otimes_{A_0} A_{!!}$ is almost zero. We want to show $M = 0$. As $A_0 \to A_{!!}$ is faithfully flat, we may filter $M$ to reduce to the case where $M = A_0/I$ for some ideal $I \subset A_0$. The hypothesis $M \otimes_{A_0} A_{!!}$ being almost zero then translates to $\pi^{\frac{1}{p^n}} \in IA_{!!}$ for all $n \geq 0$. But this implies $p = \pi^p v^{-1} \in I^{p^n} A_{!!}$ for all $n \geq 0$. By faithful flatness of $A_0 \to A_{!!}$, we must have $p \in I^{p^n}$ for all $n \geq 0$. But Krull's intersection theorem implies $\cap_n I^{p^n} = 0$ if $I$ is non-trivial. As $p \neq 0$ on $A_0$, we must therefore have $I = A_0$, and thus $M = 0$ as wanted.
\end{proof}

Finally, we recall a slightly non-standard consequence of the Artin-Rees lemma.

\begin{lemma}
\label{lem:lim1vanishing}
Let $R$ be a noetherian ring equipped with an ideal $I$. For any pair $M,N$ of finitely generated $R$-modules, the pro-$R$-modules $\{\mathrm{Hom}_R(M,N)/I^n\}_{n \geq 1}$ and $\{\mathrm{Hom}_R(M,N/I^nN)\}_{n \geq 1}$ are pro-isomorphic via the natural map. In particular, $\lim^1 \mathrm{Hom}_R(M,N/I^nN) = 0$.
\end{lemma}
\begin{proof}
We shall use the Artin-Rees lemma in the following form: the functor $P \mapsto \{P/I^n\}$ is an exact functor from finitely generated $R$-modules $P$ to pro-$R$-modules. To apply this, pick a presentation 
\[ F_1 \to F_0 \to M \to 0\]
with $F_i$ being finite free. Applying $\mathrm{Hom}_R(-,N)$ gives an exact sequence
\[ 0 \to \mathrm{Hom}_R(M,N) \to \mathrm{Hom}_R(F_0,N) \to \mathrm{Hom}_R(F_1,N).\]
The previously mentioned form of the Artin-Rees lemma then yields an exact sequence of pro-$R$-modules of the form
\[ 0 \to \{\mathrm{Hom}_R(M,N)/I^n\}_{n \geq 1} \to \{\mathrm{Hom}_R(F_0,N)/I^n\}_{n \geq 1} \to \{\mathrm{Hom}_R(F_1,N)/I^n\}_{n \geq 1}.\]
Repeating this analysis using the functor $\mathrm{Hom}_R(-,N/I^nN)$ instead gives an exact sequence of pro-$R$-modules
\[ 0 \to \{\mathrm{Hom}_R(M,N/I^nN)\}_{n \geq 1} \to \{\mathrm{Hom}_R(F_0,N/I^nN)\}_{n \geq 1} \to \{\mathrm{Hom}_R(F_1,N/I^nN)\}_{n \geq 1}.\]
Comparing the sequences yields the lemma as $\mathrm{Hom}_R(F_i,N/I^nN) \simeq \mathrm{Hom}_R(F_i,N)/I^n$ since $F_i$ is finite free.
\end{proof}

We can now prove the promised theorem.

\begin{theorem}
\label{thm:DSC}
Let $A_0$ be a regular ring. The map $A_0 \to B_0$ of $A_0$-modules is split.
\end{theorem}
\begin{proof}
We may assume that $A_0$ is a noetherian regular local ring of mixed characteristic $(0,p)$. Choose $g \in A_0$ coprime to $p$ such that $A_0 \to B_0$ is finite \'etale after inverting $pg$. Let $A_0 \to A_{\infty,0}$ be the extension provided by Proposition~\ref{prop:RegularPerfectoidFaithfullyFlat}, and let $A_{\infty,0} \to A_\infty$ be the extension resulting from applying Theorem~\ref{thm:AdjoinRootsDisc} to the integral perfectoid ring $A_{\infty,0}$ equipped with the element $g$.

Consider the canonical exact triangle 
\[ A_0 \to B_0 \to Q_0\]
of $A_0$-modules. The boundary map $\alpha_0 \in \mathrm{Hom}_{A_0}(Q_0, A_0[1])$ is the obstruction to this sequence being split. We would like this show this obstruction vanishes. We change subscript to denote derived base change to either $A_{\infty,0}$ or $A_{\infty}$; for example, $Q_{\infty,0} := Q_0 \otimes_{A_0}^L A_{\infty,0}$, $\alpha_\infty := \alpha_0 \otimes_{A_0}^L A_\infty$, etc.

First, it suffices to show that $\alpha_0/p^m \in \mathrm{Hom}_{A_0}(Q_0, A_0/p^m[1])$ vanishes for all $m \gg 0$. Indeed, we have 
\[ \mathrm{Hom}_{A_0}(Q_0, A_0[1]) \simeq \lim_m \mathrm{Hom}_{A_0}(Q_0, A_0/p^m[1]),\]
as $A_0$ is $p$-adically complete and $\{\mathrm{Hom}_{A_0}(Q_0, A_0/p^m)\}_{m \geq 1}$ has vanishing $\lim^1$ by Lemma~\ref{lem:lim1vanishing}.

Choose  $m \geq 3$ such that $\alpha_0/p^m \neq 0$, so $\mathrm{Ann}_{A_0/p^m}(\alpha_0/p^m) \neq A_0/p^m$; if no such $m$ exists, then $\alpha_0/p^m = 0$ for all $m$, so we are done. Otherwise, by Krull's theorem, there exists $k \geq 0$ such that $p^2 g \notin \big(\mathrm{Ann}_{A_0/p^m}(\alpha_0/p^m)\big)^{p^k}$; here we use that $m \geq 3$ and that $0 \neq g \in A_0/p$. As both $A_0/p^m \to A_{\infty,0}/p^m$ and $A_{\infty,0}/p^m \to A_\infty/p^m$ are almost faithfully flat with respect to $p^{\frac{1}{p^\infty}}$, we get $pg \notin \big(\mathrm{Ann}_{A_\infty/p^m}(\alpha_\infty/p^m)\big)^{p^k}$, so $(pg)^{\frac{1}{p^k}} \notin \mathrm{Ann}_{A_\infty/p^m}(\alpha_\infty/p^m)$; here we lose a power of $p$ in passing to almost mathematics. It is thus enough (via contradiction) to show that $\alpha_\infty/p^m \in \mathrm{Hom}_{A_\infty}(Q_\infty, A_\infty/p^m[1])$ is almost zero with respect to $(pg)^{\frac{1}{p^\infty}}$. 

Consider the tower $\{A_\infty \langle \frac{p^n}{g} \rangle\}$ from \S \ref{sec:QuantRET}.  As $g$ divides $p^n$ in $A_\infty \langle \frac{p^n}{g} \rangle$, the base change $A_\infty \langle \frac{p^n}{g} \rangle \to B_0 \otimes^L_{A_0} A_\infty \langle \frac{p^n}{g} \rangle$ of $A_0 \to B_0$ is finite \'etale after inverting $p$. Almost purity \cite[Theorem 7.9 (iii)]{ScholzePerfectoidSpaces} then implies that this base change can be dominated by an almost finite \'etale cover of $A_\infty \langle \frac{p^n}{g} \rangle$, and is thus almost split with respect to $p^{\frac{1}{p^\infty}}$ (see \cite[Lemma 2.7]{BhattAlmostDSC}). The same then holds modulo $p^m$, so the image of $\alpha_\infty/p^m$ under 
\[ \mathrm{can}:\mathrm{Hom}_{A_\infty}(Q_\infty, A_\infty/p^m[1]) \to \lim_n \mathrm{Hom}_{A_\infty}(Q_\infty, A_\infty \langle \frac{p^n}{g} \rangle/p^m[1])\]
is almost zero with respect to $p^{\frac{1}{p^\infty}}$. It is now enough  to show that the above map is an almost isomorphism with respect to $(pg)^{\frac{1}{p^\infty}}$. By Theorem~\ref{thm:QuantRET}, the only obstruction is $\lim^1$ of $\{\mathrm{Hom}_{A_\infty}(Q_\infty, A_\infty \langle \frac{p^n}{g} \rangle/p^m)\}_{n \geq 1}$. This pro-system is almost-pro-isomorphic to a constant pro-system by Lemma~\ref{lem:PAIAlinear} and Theorem~\ref{thm:QuantRET}, so we are done by Lemma~\ref{lem:PAIAlmostIso}.
\end{proof}

\begin{remark}
The proof given above goes through for any noetherian ring $A_0$ that admits a faithfully flat extension which is integral perfectoid. Thus, one may ask: does this condition characterize regularity? In other words, is there a $p$-adic analog of Kunz's theorem characterizing regularity in characteristic $p$ as the flatness of Frobenius? A positive answer to this question will appear in forthcoming work of the author with Iyengar and Ma.
\end{remark}

\section{The derived direct summand conjecture}
\label{sec:DDSC}

The goal of this section is to prove Theorem~\ref{thm:DDSCIntro}

\begin{theorem}
\label{thm:DDSC}
Let $A_0$ be a regular noetherian ring, and let $f_0:X_0 \to \mathrm{Spec}(A_0)$ be a proper surjective map. Then the map $A_0 \to R\Gamma(X_0, \mathcal{O}_{X_0})$ splits in $D(A_0)$.
\end{theorem}

\begin{proof}
We may assume $A_0$ is a regular local ring. By taking the closure of a suitable generically defined multisection, we may assume that $X_0$ is integral and $f_0$ is generically finite. Then we can choose $g \in A_0$ coprime to $p$ such that $f_0$ is finite \'etale after inverting $pg$. Construct $A_{\infty,0}$ and $A_\infty$ as Theorem~\ref{thm:DSC}. Repeating the argument in the proof of Theorem~\ref{thm:DSC}, we must show that for fixed $m,n \geq 1$, the map 
\[ A_\infty \langle \frac{p^n}{g} \rangle/p^m \to R\Gamma(X_0,\mathcal{O}_{X_0}) \otimes^L_{A_0} A_\infty \langle \frac{p^n}{g} \rangle/p^m\]
is almost split with respect to $p^{\frac{1}{p^\infty}}$. As the base change $X_0 \times_{\mathrm{Spec}(A_0)} \mathrm{Spec}(A_\infty \langle \frac{p^n}{g} \rangle) \to \mathrm{Spec}(A_\infty \langle \frac{p^n}{g} \rangle)$ is proper and finite \'etale after inverting $p$ (as $g$ divides $p^n$ on the base), Proposition~\ref{prop:DDSCAdmissibleAlteration} and a diagram chase finish the proof.
 \end{proof}

The following special case of Theorem~\ref{thm:DDSC} is the crucial one:

\begin{proposition}
\label{prop:DDSCAdmissibleAlteration}
Let $A$ be an integral perfectoid $K^\circ$-algebra, and set $S = \mathrm{Spec}(A)$. Let $f:Y \to S$ be a proper morphism such that $f[\frac{1}{p}]$ is finite \'etale. Then $A \to R\Gamma(Y,\mathcal{O}_Y)$ is almost split.
\end{proposition}
\begin{proof}
Let $B = H^0(Y, \mathcal{O}_Y)$, so $B$ is an integral extension of $A$ which is finite \'etale after inverting $p$, and $Y$ is naturally a $B$-scheme. Almost purity \cite[Theorem 7.9 (iii)]{ScholzePerfectoidSpaces} gives a map $B \to C$ which is an isomorphism after inverting $p$ such that the induced map $A \to C$ is an almost finite \'etale cover. In particular, $C$ is integral perfectoid, and $A \to C$ is almost split. Thus, on replacing $A$ with $C$ and $Y$ with $Y \otimes_B C$, we may assume that $f[\frac{1}{p}]$ is an isomorphism. But then the $p$-adic completion $\widehat{f}$ of $f$ can be dominated by an admissible blowup of $S$. Set $S_\eta  = \mathrm{Spa}(A[\frac{1}{p}])$ to be the associated affinoid perfectoid space, so the natural map $(S_\eta, \mathcal{O}_{S_\eta}^+) \to S$ factors\footnote{This assertion is implicit in the description of adic spaces given in \cite[\S 15.4]{GabberRameroFART}, and can be explained roughly as follows. It is enough to define the map when $Y$ is the blowup of $S$ at an ideal $I = (f_1,...,f_n,g)$ with $f_n = p^N$. The map $S_\eta \to Y$ on topological spaces is defined immediately via the valuative criterion of properness. The preimage of the chart $Y_g \subset Y$ where $g \mid f_i$ is the rational subset $S_\eta\Big(\frac{f_1,...,f_n}{g}\Big) \subset S_\eta$. Moreover, the sections of $\mathcal{O}_{S_\eta}^+$ on $S_\eta\Big(\frac{f_1,...,f_n}{g}\Big)$ are a suitable completion of the integral closure in $A[\frac{1}{g}]$ of $H^0(Y_g,\mathcal{O}_Y)$, which gives the map on sheaves.}  through every admissible blowup of $\widehat{S}$. In particular, it factors as
\[ (S_\eta, \mathcal{O}^+_{S_\eta}) \to Y \to S.\]
Taking cohomology of the structure sheaf gives 
\[ A \xrightarrow{b} R\Gamma(Y, \mathcal{O}_Y) \xrightarrow{a} R\Gamma(S_\eta, \mathcal{O}^+_{S_\eta}).\]
Now $a \circ b$ is an almost isomorphism by Scholze's vanishing theorem \cite[Proposition 6.14]{ScholzePerfectoidSpaces}, so $b$ is almost split. 
\end{proof}



\renewcommand{\baselinestretch}{1.4}

\end{document}